\documentclass[twoside]{amsart}
\usepackage[english]{babel}

\usepackage{graphicx}
\usepackage{mathtools}		
\usepackage{mathabx}		
\usepackage{indentfirst}	
\usepackage{amsmath}
\usepackage{amssymb} 
\usepackage{amsthm}
\usepackage{thmtools}
\usepackage{enumitem}		
\usepackage[backref=page,colorlinks=true]{hyperref}	
\usepackage[capitalise]{cleveref} 

\usepackage{tikz}
\usepackage[font=footnotesize]{caption}

\newcommand{\Hy}{\mathbb{H}}
\newcommand{\C}{\mathbb{C}}
\newcommand{\R}{\mathbb{R}}
\renewcommand{\epsilon}{\varepsilon}

\newcommand{\groprod}[3]{\left<#1|#2\right>_{#3}}

\newcommand{\corchete}[1]{\left\{{#1}\right\}}
\newcommand{\ov}[1]{\overline{#1}}

\newcommand{\SL}[2]{\mathrm{SL}_{#1}(#2)}

\newcommand{\CAT}[1]{\textup{CAT}(#1)}

\DeclareMathOperator{\Ima}{Im}

\DeclareMathOperator{\arccosh}{arccosh}

\usepackage{color}
\usepackage[colorlinks=true]{hyperref}
\usepackage{fancyhdr}
\usepackage[left=4.5cm,top=3.5cm,right=4.5cm,bottom=2.5cm]{geometry}

\renewcommand*{\backref}[1]{}
\renewcommand*{\backrefalt}[4]{\quad \tiny 
  \ifcase #1 (\textbf{NOT CITED.})%
  \or    (Cited on page~#2.)%
  \else   (Cited on pages~#2.)%
  \fi}

\makeatletter

\def\MRbibitem{\@ifnextchar[\my@lbibitem\my@bibitem}

\def\mybiblabel#1#2{\@biblabel{{\hyperref{http://www.ams.org/mathscinet-getitem?mr=#1}{}{}{#2}}}}

\def\myhyperanchor#1{\Hy@raisedlink{\hyper@anchorstart{cite.#1}\hyper@anchorend}}

\def\my@lbibitem[#1]#2#3#4\par{%
  \item[\mybiblabel{#2}{#1}\myhyperanchor{#3}\hfill]#4%
  \@ifundefined{ifbackrefparscan}{}{\BR@backref{#3}}%
  \if@filesw{\let\protect\noexpand\immediate
    \write\@auxout{\string\bibcite{#3}{#1}}}\fi\ignorespaces%
}

\def\my@bibitem#1#2#3\par{%
  \refstepcounter\@listctr
  \item[\mybiblabel{#1}{\the\value\@listctr}\myhyperanchor{#2}\hfill]#3%
  \@ifundefined{ifbackrefparscan}{}{\BR@backref{#2}}%
  \if@filesw\immediate\write\@auxout
    {\string\bibcite{#2}{\the\value\@listctr}}\fi\ignorespaces%
}

\makeatother

\newtheorem{teo}{Theorem}[section]
\newtheorem{prop}[teo]{Proposition}
\newtheorem{lema}[teo]{Lemma}
\newtheorem{coro}[teo]{Corollary}
\theoremstyle{remark}
\theoremstyle{definition}
\newtheorem{ej}[teo]{Example}
\newtheorem{defi}[teo]{Definition}

\begin{document}

\title{\textbf{The Avalanche Principle and Negative Curvature}}

\author{\small{Eduardo Oreg\'on-Reyes}}

\markboth{E Oreg\'on Reyes}{Nombreartículo}

\date{}
\maketitle

\begin{abstract}
We use the geometric structure of the hyperbolic upper half plane to provide a new proof of the \emph{Avalanche Principle} introduced by M. Goldstein and W. Schlag in the context of $\SL{2}{\R}$ matrices. This approach allows to interpret and extend this result to arbitrary $\CAT{-1}$ metric spaces. Through the proof, we deduce a polygonal Schur theorem for these spaces.
\end{abstract}



\section{Introduction}

Lyapunov exponents play a major role in the theory of dynamical systems, codifying the asymptotic behavior of a sequence of composition of linear maps. In particular, the top Lyapunov exponent describes the evolution of the norms of matrix products. For matrices taking values in $\SL{2}{\R}$, M. Goldstein and W. Schlag introduced the \emph{Avalanche Principle} (AP) \cite{gs}, which is a quantitative sufficient condition for the operator norm $\|A_{N}\cdots A_{1}\|$ to being similar to the product $\|A_{N}\|\cdots\|A_{1}\|$. Since then, several higher dimensional versions and refinements have appeared in the literature, being an important tool to prove continuity of Lyapunov exponents for linear cocycles coming from Schr\"odinger operators (see e.g. \cite{bour, bourji, duaklqp, dk18, sch}). The following is the version for $\SL{2}{\R}$ due to Duarte and Klein \cite[Thm.~4.1]{duakl}:
\begin{teo}[AP in $\SL{2}{\R}$]\label{APM} There exist constants $c_0,c_1>0$ so that if $0 <\epsilon <1$ and $0<\kappa\leq c_0\epsilon^{2}$, then for every chain of matrices $A_1,\dots,A_{n} \in \SL{2}{\R}$ satisfying
\begin{equation*}
    \|A_{j}\|\geq \kappa^{-2} \hspace{2mm}\text{ for }1\leq j\leq n, \hspace{4mm} \text{and} \hspace{4mm}
    \displaystyle\frac{\|A_{j}A_{j-1}\|}{\|A_{j}\|\|A_{j-1}\|}\geq \epsilon \hspace{2mm}\text{ for } 2\leq j\leq n,
\end{equation*}
we have
\begin{equation}
\left| \log {\|A_{n}\cdots A_{1}\|} +\sum_{i=2}^{n-1}{\log{\|A_{i}\|}}-\sum_{i=2}^{n}{\log{\|A_{i}A_{i-1}\|}}\right|\leq c_1\frac{n\kappa}{\epsilon^{2}}. \label{concAP}
\end{equation}
\end{teo}
Here $\|A\|$ denotes the Euclidean operator norm (largest singular value) of the matrix $A$.

Note that the hypothesis only depends on the norms of the matrices $A_1,\dots,A_{n}$ and $A_2A_1,\dots,A_{n} A_{n-1}$, so we can think of the AP as a \emph{local to global} principle for norms of matrix products. In fact, these hypotheses imply \emph{uniform hyperbolicity} of infinite sequences of $2\times 2$ matrices \cite{zh}. 

\subsection{A version for the hyperbolic plane} The assumptions and conclusions of the previous theorem have natural interpretations in terms of hyperbolic geometry. Consider the upper half plane $\Hy^{2}=\corchete{z \in \C \colon \Ima{z}>0}$ endowed with the Riemannian metric $ds^{2}=dz^{2}/\Ima(z)^{2}$. The induced distance $d$ on $\Hy^{2}$ takes the form \cite[p.~130]{bear}
\begin{equation}
d(z_1,z_2)=  \arccosh \left(1+\frac{ |z_1-z_2|^2 }{ 2 \textup{Im}(z_1)\textup{Im}(z_1) } \right)=2\log{\left(\frac{|z_1-z_2|+|z_1-\overline{z_2}|}{2\sqrt{\textup{Im}(z_1)\textup{Im}(z_1)}}\right)}
.\label{disthip}
\end{equation}
We also have the natural isometric action of $\SL{2}{\R}$ on $\Hy^2$ by fractional linear transformations: 
\begin{equation}
A=\begin{pmatrix}
a & b \\
c & d \\
\end{pmatrix}
\mapsto \tilde{A}z=\
\displaystyle\frac{az+b}{cz+d}. \notag
\end{equation}
The relation between the operator norm $\|A\|$ of $A$ and its action $\tilde{A}$ on $\Hy^2$ is given by the formula \cite[Prop.~2.1]{yogro}
\begin{equation*}
d(\tilde{A}\mathbf{i},\mathbf{i})=2 \log{\|A\|}.
\end{equation*}
So, if we define $x_0=\mathbf{i}$ and $x_{j}=\tilde{A}_{n}\cdots  \tilde{A}_{n-j+1}\mathbf{i}$ for $1\leq j\leq n$, then the left hand side of \eqref{concAP} translates to the following definition:
\begin{defi}The \emph{tension} of a chain of points $x_0,\dots,x_n\in \Hy^{2}$ is the number:
\begin{equation*} \label{rift}
\tau(x_0,\dots,x_n)=\sum_{j=1}^{n-1}{d(x_{j-1},x_{j+1})}-\sum_{j=2}^{n-1}{d(x_{j-1},x_{j})}-d(x_0,x_n).
\end{equation*}
\end{defi}
\begin{ej}
If $x_0,\dots,x_n$ lie in a hyperbolic geodesic in $\Hy^{2}$ in that order, then $\tau(x_0,\dots,x_n)=0$.
\end{ej}

\begin{ej}\label{circulo}
Let $n\geq 4$, and consider a regular hyperbolic $n$-gon $x_1,\dots,x_n \in \Hy^2$ inscribed in a hyperbolic circle of radius $r$. Defining $x_0=x_n$, the tension of the chain $x_0,x_1,\dots,x_n$ is $$\tau(x_0,\dots,x_n)=(n-1)d(x_0,x_2)-(n-2)d(x_0,x_1)\geq d(x_0,x_2).$$
Thus $\frac{\tau(x_0,\dots,x_{n})}{n}$ tends to infinity when $r$ tends to infinity, and there are chains with arbitrarily large tension when compared to their lengths.
\end{ej}
\begin{figure}[hbt]
\begin{tikzpicture}[scale=0.4]
\draw (2.21338,2.7755) node{\small{$\bullet$}}; 
\draw (-0.78994,3.46099) node{\small{$\bullet$}};
\draw (-3.19843,1.54028) node{\small{$\bullet$}};
\draw (-3.19843,-1.54028) node{\small{$\bullet$}};
\draw (-0.78994,-3.46099) node{\small{$\bullet$}}; 
\draw (2.21338,-2.7755) node{\small{$\bullet$}};
\draw (3.55,0) node{\small{$\bullet$}}; 
\draw (2.21338,2.7755) arc (155.28435:256.14421:1.998177)  ;
\draw (-0.78994,3.46099) arc (206.712921429
:307.572781429:1.998177)  ;
\draw (-3.19843,1.54028) arc (258.141492858:359.001352858
:1.998177)  ;
\draw (-3.19843,-1.54028) arc (309.570064287:410.429924287:1.998177)  ;
\draw (-0.78994,-3.46099) arc (360.998635716:461.858495716:1.998177)  ;
\draw (2.21338,-2.7755) arc (412.427207145:513.287067145:1.998177)  ;
\draw (3.55,0) arc (463.855778574
:564.715638574:1.998177)  ;
\draw[dashed] (0,0) circle (3.55) ;
\end{tikzpicture}
	\caption{A hyperbolic heptagon with large tension.}
\end{figure}
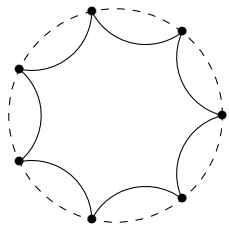

As the previous examples suggest, for $\frac{|\tau(x_0,\dots,x_{n})|}{n}$ to be small we need some kind of control on the chain $x_0,x_1,\dots,x_n$, making it close to lie in a geodesic. In $\Hy^2$, a sufficient condition for this is that the points $x_0,\dots,x_n$ lie in that order in a curve of constant geodesic curvature less than 1 w.r.t.~the hyperbolic metric (for a detailed explanation, see \cite[Sec.~2.3]{hub}). By Example \ref{circulo} this condition is in some sense necessary, since contrary to what happens in Euclidean geometry, a curve in $\Hy^2$ of constant geodesic curvature $k$ is closed (i.e.~is a hyperbolic circle) if and only if $k>1$ \cite[Exe.~2.3.7]{hub}. 
 We control the chains in $\Hy^2$ according to the following definition:   



\begin{defi}\label{defigood}
A pair $(a,b)\in \R^{2}$ is \emph{good} if $a,b\geq 0$ and 
\begin{equation}\label{goodp}
\sinh(a-b)>2\sinh(a/2).
\end{equation}
For such a pair, we say that a chain $x_0,x_1,\dots,x_n$ of points in $\Hy^2$ is $(a,b)$-\emph{good} if
\begin{equation}\label{goodc}
d(x_{j+1},x_j)\geq a \hspace{1.5mm}\text{ for }0\leq j\leq n-1,\hspace{1.8mm}\text{and}\hspace{1.8mm} \groprod{x_{j-1}}{x_{j+1}}{x_j}\leq b \hspace{1.5mm}\text{ for }1\leq j\leq n-1,
\end{equation}
where $\groprod{x}{y}{z}:=\frac{d(x,z)+d(z,y)-d(x,y)}{2}$ is the \emph{Gromov product}.
\end{defi}
\hspace{-4.4mm}Condition \eqref{goodp} is natural, since for an orientation-preserving isometry $f$ of $\Hy^2$ with
\begin{equation}\label{cchain}
 d(f\mathbf{i},\mathbf{i})=a, \hspace{2mm}\text{ and }\groprod{f^{2}\mathbf{i}}{\mathbf{i}}{f\mathbf{i}}=b,
\end{equation}
the chain of points $\mathbf{i},f\mathbf{i},f^2\mathbf{i},\dots$ lies in a curve of constant geodesic curvature less than 1 if and only if \eqref{goodp} holds. In that case the \emph{stable length} $d^{\infty}(f)=\inf_{n\geq1 }{\frac{d(f^{n}\mathbf{i},\mathbf{i})}{n}}$ of $f$ is positive and satisfies \cite[Cor.~3]{yomt}: 
\begin{equation*}
\sinh(a-b)=2\sinh(a/2)\cosh(d^{\infty}(f)/2).
\end{equation*}
Gromov product is also natural. We have $\groprod{x}{y}{z}\geq0 $, with equality if and only if $x,y,z$ lie in a geodesic with $z$ between $x$ and $y$. Moreover, when $d(x,z)$ and $d(z,y)$ are large, $\groprod{x}{y}{z}$ is essentially a function of the angle determined by $x,y,z$ with vertex at $z$, w.r.t.~the hyperbolic metric. So when $a$ is large, condition \eqref{goodc} may be regarded as an \emph{angular bound}.

Let $x_0,\dots,x_n \in \Hy^{2}$ be a chain of the form $x_j=f^{j}\mathbf{i}$ for some orientation-preserving isometry $f$ of $\Hy^{2}$ satisfying \eqref{cchain}. A chain of this form is called $(a,b)$-\emph{canonical}. Note that two $(a,b)$-canonical chains of the same length are isometric, so we refer to any of these as \emph{the} $(a,b)$-canonical chain. 
In what follows, two chains $x_0,\dots,x_n$ and $y_0,\dots,y_n$ are considered the same if $y_j=fx_j$ for some isometry $f$ of $\Hy^2$. The heuristic is that among all chains which are $(a,b)$-good, those that are the farthest from a geodesic are the $(a,b)$-canonical ones.

For a good pair $(a,b)$, consider the numbers $\lambda>1$ and $0<\varphi\leq \pi/2$ given by 
\begin{equation}
\cosh(\log(\lambda)/2)=\frac{\sinh(a-b)}{2\sinh(a/2)}, \hspace{3mm}
\csc(\varphi)\sinh(\log(\lambda)/2)=\sinh(a/2). \label{parameters}
\end{equation}
For these quantities the chain $x_0,\dots,x_n$ given by $x_j=\lambda^{j}e^{\mathbf{i}\varphi}$ is $(a,b)$-canonical (see Corollary \ref{saco} below). Also, note that $\lambda=e^{d^{\infty}(f)}$ for $fz=\lambda z$, and that the curve $t\mapsto te^{\mathbf{i}\varphi}$ has constant geodesic curvature equal to $\cos(\varphi)<1$. We call the numbers $\lambda$ and $\varphi$ the \emph{translation number} and \emph{curvature angle} of the pair $(a,b)$, respectively, and we say that a chain $x_0,x_1,\dots,x_n$ is $\varphi$-\emph{good} if it is $(a,b)$-good of a good pair $(a,b)$ with curvature angle $\varphi$.
\begin{figure}[hbt]
\begin{tikzpicture}[scale=0.95]
\draw[dashed] (0,0) -- (9.028125,2.1515625) ;
\draw (-1.5,0) -- (9.128125,0);
\draw (2.14,0.51) node[below] {\small{$x_{0}$}}; 
\draw (2.14,0.51) node{\small{$\bullet$}};
\draw (3.21,0.765) node[below] {\small{$x_{1}$}};
\draw (3.21,0.765) node{\small{$\bullet$}};
\draw (4.815,1.1475) node[below] {\small{$x_{2}$}}; 
\draw (4.815,1.1475) node{\small{$\bullet$}};
\draw (7.2225,1.72125) node[below] {\small{$x_{3}$}}; 
\draw (7.2225,1.72125) node{\small{$\bullet$}};
\draw (1.8,0) arc (0:13.40456:1.8)  ;
\draw (3.21,0.765) arc (63.0046607:143.4084:0.855552153)  ;
\draw (4.815,1.1475) arc (63.0046607:143.4084:1.2833282)  ;
\draw (7.2225,1.72125) arc (63.0046607:143.4084:1.9249923)  ;
\draw (1.55,0.19) node {\small{$\varphi$}};
\end{tikzpicture}
	\caption{A canonical chain with $\lambda=1.5$, $\varphi=13.404^{\circ}$.}
\end{figure}
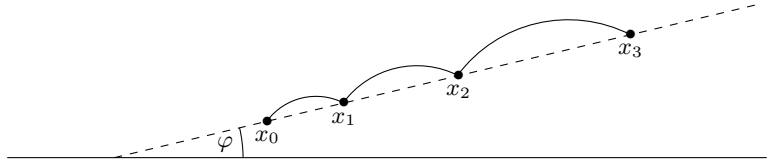

We have enough notation for stating our first main result, the AP in the upper half plane:
\begin{teo}[Hyperbolic Avalanche Principle]\label{AP}Let $(a,b)$ be a good pair, and let $x_0,x_1,\dots,x_n\in \Hy^{2}$ be an $(a,b)$-good chain. Then
\begin{equation*}\label{hap}
\left|\tau(x_0,\dots,x_n) \right|\leq (n-2)\frac{2}{\lambda-1}, 
\end{equation*}
where $\lambda$ is the translation number of $(a,b)$.
\end{teo}
The conclusion of the AP is then that the quantity $\frac{|\tau(x_0,\dots,x_n)|}
{n}$ is small when the chain $x_0, x_1,\dots,x_n$ is close to lie in a geodesic in the sense of Definition \ref{defigood}, and it implies Theorem \ref{APM}. Indeed, the inequality \cite[Thm.~1.1]{yogro} $$d^{\infty}(f)\geq d(f^{2}\mathbf{i},\mathbf{i})-d(f\mathbf{i},\mathbf{i})-2\log{2}$$ holds for any orientation-preserving isometry $f$ of $\Hy^2$, and for $f$ satisfying \eqref{cchain} it turns out to be equivalent to $\lambda \geq \frac{1}{4}e^{a-2b}$. So, if $c>1$ and $a-2b> \log(4)+\log\left(\frac{c}{c-1}\right)>\log(4)$, then $(a,b)$ is a good pair and  
\begin{equation*}
\left|\tau(x_0,\dots,x_n) \right|\leq (n-2)\frac{2}{\lambda-1} \leq 8c(n-2)e^{2b-a}, 
\end{equation*}
and we recover \eqref{concAP} with $\kappa=e^{-a}$, $\epsilon=e^{-b}$, $c_1=4c>4$ and $c_0=\frac{c-1}{4c}$.
Condition \eqref{goodp} is more flexible than Duarte-Klein hypotheses for AP, since it also includes chains $x_0,\dots,x_n$ with $d(x_j,x_{j-1})$ arbitrarily close to $0$. This allow us to conclude results of continuous nature, as we see below. 

\subsection{AP for $\CAT{-1}$ spaces and Schur Theorem} The notions of tension and Gromov product are valid for arbitrary metric spaces, so the question is for which of them an Avalanche Principle holds. Natural candidates are \emph{$\CAT{-1}$ metric spaces}, whose local and global geometry are more negatively curved than the geometry of $\Hy^2$. These are metric spaces whose geodesic triangles are \emph{thinner} than the respective geodesic triangles in $\Hy^2$ (see Section \ref{sAPCAT} for a detailed definition). Examples of such spaces include metric trees, and complete simply connected Riemannian manifolds with sectional curvature bounded above by $-1$ with the induced Riemannian distance \cite[Ch.~II, Thm.~IA.6]{bri}. 

Similarly to the definition given for $\Hy^{2}$, a chain $x_0,\dots,x_n$ of points in a metric space $X$ is good if there is a good pair $(a,b)$ such that the points $x_0,\dots,x_n$ satisfy \eqref{goodc} with the corresponding distance on $X$. For such chains there is an essentially unique convex comparison chain $\ov{x}_0,\dots,\ov{x}_n\in \Hy^{2}$ such that $d(x_j,x_{j-1})=d(\ov{x}_j,\ov{x}_{j-1})$ for $1\leq j\leq n$ and $\groprod{x_{j-1}}{x_{j+1}}{x_j}=\groprod{\ov{x}_{j-1}}{\ov{x}_{j+1}}{\ov{x}_j}$ for $1\leq j\leq n-1$ (see Definitions \ref{condef} and \ref{compdef}). 

Once we know that AP holds for convex chains in $\Hy^2$, the Avalanche Principle for $\textup{CAT}(-1)$ spaces follows from the next theorem:
\begin{teo}\label{catcom}Let $X$ be a $\textup{CAT}(-1)$ space, and consider a good chain $x_0,x_1,\dots,x_n$ in $X$ with respective comparison chain $\overline{x}_0,\overline{x}_1,\dots,\overline{x}_n$ in $\Hy^{2}$. Then 
\begin{equation*}
|\tau(x_0,x_1,\dots,x_n)|\leq \tau(\overline{x}_0,\overline{x}_1,\dots,\overline{x}_n).
\end{equation*}
\end{teo}

\begin{coro}[$\CAT{-1}$ AP]\label{catapc}
Theorem \ref{AP} also holds for $\CAT{-1}$ spaces.
\end{coro}

The inequality $\tau(x_0,x_1,\dots,x_n)\leq \tau(\overline{x}_0,\overline{x}_1,\dots,\overline{x}_3)$ in Theorem \ref{catcom}, which is equivalent to $d(x_0,x_n)\leq d(\ov{x}_0,\ov{x}_n)$, was proved by C. Epstein in $\Hy^{3}$ \cite{eps} by an elaborate argument, and by A. Granados when $X$ is a complete simply connected Riemannian manifold with sectional curvature bounded above by $-1$ \cite{gr}, under similar assumptions. Both authors used this inequality to prove extensions of the classical Schur comparison theorem for plane curves \cite[p.~36]{chern}:
\begin{teo}[Extended Schur]\label{schur}
Let $X$ be a complete simply connected Riemannian manifold with sectional curvature bounded above by $-1$. Suppose that $f$ is a curve in $X$ with length $L$ and $g$ is a simple curve in $\Hy^2$ with the same length $L$ that together with its chord bounds a convex region of the upper half plane. Suppose that the geodesic curvatures satisfy
$k_f(s) \leq k_g(s)$ , where $s$ is the common arc-length parameter. Then the length of the chord of $f$ is greater than or equal to the length of the chord of $g$.
\end{teo}
Therefore, Theorem \ref{catcom} implies the previous versions for the case of good chains, and hence Schur theorem \ref{schur} for curves with geodesic curvatures $k_f(s) \leq k_g(s) \leq 1$. See also \cite[Thm~2.3.13]{hub} for a result of similar spirit.

\medskip

\hspace{-4.3mm}\textbf{Organization of the paper:}
Section \ref{prelimhip} presents the main properties of the hyperbolic plane $\Hy^2$ that we use throughout the paper. In Section \ref{concha} we reduce the study of convex chains to the ones contained in curves of constant geodesic curvature. We use this reduction in Section \ref{APCC} and prove Theorem \ref{AP} for the case of convex chains. Section \ref{sAPCAT} deals with the non convex case, as well as with Theorem \ref{catcom}.




\section{Preliminaries of Hyperbolic Geometry}\label{prelimhip}

We start with some basic properties of the upper half plane $\Hy^2$. This is a \emph{geodesic metric space} in the sense that every two points $x,y\in \Hy^2$ can be joined by an arc isometric to a closed interval of length $d(x,y)$. This arc is unique and is denoted by $xy$. For every three distinct points $x,y,z$ in $\Hy^2$, the Riemannian angle between the arcs $zx$ and $zy$ is denoted by $\angle_{z}(x,y)$. The relation between angles and distances is given by the \emph{hyperbolic laws of cosines and sines}:
\begin{prop}
For a geodesic triangle in $\Hy^2$  with sides $a$, $b$ and $c$ and opposite angles $\alpha$, $\beta$ and $\gamma$: 

\emph{Law of Cosines}:
\begin{equation}
\cosh(c)=\cosh(a)\cosh(b)-\sinh(a)\sinh(b)\cos(\gamma).  \tag{LC}   \label{lcos}
\end{equation}

\vspace{-1ex}
\emph{Law of Sines}:
\begin{equation}
\frac{\sin(\alpha)}{\sinh(a)} =\frac{\sin(\beta)}{\sinh(b)}=\frac{\sin(\gamma)}{\sinh(c)}. \tag{LS} \label{lsin} 
\end{equation}
\end{prop} 
A quadrilateral with vertices $x,y,z,w \in \Hy^2$ such that $\angle_{x}(w,y)=\angle_{y}(x,z)=\pi/2$ and $d(x,w)=d(y,z)$ is called a \emph{Saccheri quadrilateral}. As a consequence of the hyperbolic trigonometric laws we obtain:
\begin{coro}
If $x,y,z,w \in \Hy^2$ form a Saccheri quadrilateral with $\angle_{x}(w,y)=\angle_{y}(x,z)=\pi/2$, $d(x,w)=d(y,z)=a$, $d(x,y)=b$ and $d(z,w)=\ell$, then
\begin{equation*}
    \sinh{\left({\frac {\ell}{2}}\right)}=\cosh{(a)}\sinh {\left({\frac {b}{2}}\right)}. \label{sac}
\end{equation*}
\end{coro}
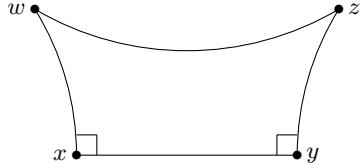
\begin{figure}[hbt]
\begin{tikzpicture}[scale=0.9]
\draw (1.63,0) node{\small{$\bullet$}}; 
\draw (-1.63,0) node{\small{$\bullet$}};
\draw (2.2452,2.15814) node{\small{$\bullet$}};
\draw (-2.2452,2.15814) node{\small{$\bullet$}};
\draw (1.63,0) node[right]{\small{$y$}}; 
\draw (-1.63,0) node[left]{\small{$x$}};
\draw (2.2452,2.15814) node[right]{\small{$z$}};
\draw (-2.2452,2.15814) node[left]{\small{$w$}};
\draw (1.63,0) -- (-1.63,0)  ;
\draw (2.2452,2.15814) arc (148.17819:180:4.0929)  ;
\draw (-1.63,0) arc (0:31.82181:4.0929)  ;
\draw (-2.2452,2.15814) arc (239.395:300.604:4.41)  ;
\draw (1.33,0) -- (1.33,0.3)  ;
\draw (1.33,0.3) -- (1.63,0.3)  ;
\draw (-1.33,0) -- (-1.33,0.3)  ;
\draw (-1.33,0.3) -- (-1.63,0.3)  ;
\end{tikzpicture}
	\caption{A Saccheri quadrilateral.}
\end{figure}
\begin{proof} Applying \eqref{lcos} to the triangles $y,w,z$ and $x,y,w$ respectively, we obtain
\begin{equation}\label{sc1}
\cosh(\ell)=\cosh(a)\cosh(d(y,w))-\sinh(a)\sinh(d(y,w))\cos(\angle_{y}(z,w)),
\end{equation}

\vspace{-6ex}
\begin{equation}\label{sc2}
\cosh(d(y,w))=\cosh(a)\cosh(b)-\sinh(a)\sinh(b)\cos(\pi/2)=\cosh(a)\cosh(b).
\end{equation}
By \eqref{lsin} applied to $x,y,w$ we also have
\begin{equation}\label{sc3}
\sinh(a)=\sinh(d(y,w))\sin(\angle_{y}(x,w))=\sinh(d(y,w))\cos(\angle_{y}(z,w)).
\end{equation}
Replacing \eqref{sc2} and \eqref{sc3} into \eqref{sc1} we have
\begin{align*}
    \cosh(\ell) &=\cosh(d(y,w))\cosh(a)-\sinh(a)[\sinh(d(y,w))\cos(\angle_{y}(z,w))] \\
    &=\cosh^{2}(a)\cosh(b)-\sinh^{2}(a)=\cosh^{2}(a)[\cosh(b)-1]+1,
\end{align*}
and hence 
$2\sinh^{2}({\ell}/2)=\cosh(\ell)-1=\cosh^{2}(a)[\cosh(b)-1]=2\cosh^{2}(a)\sinh^{2}(b/2)$. Dividing by 2 and taking square root the result follows. \end{proof}
Applying identity \eqref{disthip} we obtain $\cosh(d(\mathbf{i},e^{\mathbf{i}\varphi}))=\csc(\varphi)$ for $0<\varphi\leq \pi/2$ and $d(\mathbf{i},\lambda \mathbf{i})=\log(\lambda)$ for $\lambda>1$. In addition, the points $\mathbf{i}, e^{\mathbf{i}\varphi}$, $\lambda e^{\mathbf{i}\varphi}$ and $\lambda \mathbf{i}$ form a Saccheri quadrilateral, and the previous corollary implies
\begin{coro}\label{saco}
If $\lambda>1, 0 < \varphi \leq \pi/2$, then
 \small{$   \sinh \left(\frac{\log(\lambda)}{2}\right)=\csc(\varphi)\sinh \left(\frac{d(e^{\mathbf{i}\varphi},\lambda e^{\mathbf{i}\varphi})}{2}\right)$}.
\end{coro}
\begin{figure}[hbt]
\begin{tikzpicture}[scale=0.55]
\draw[dashed] (0,0) -- (0,5) ;
\draw[dashed] (0,0) -- (2.8960275,4.81035);
\draw(-5,0) -- (5,0) ;
\draw (1.18,1.96) node {\small{$\bullet$}}; 
\draw (2.2415,3.7231) node{\small{$\bullet$}};
\draw (0,2.2877) node {\small{$\bullet$}};
\draw (0,4.34584) node{\small{$\bullet$}};

\draw(0.3,2.2877) -- (0.3,2.5877) ;
\draw(0.3,2.5877) -- (0,2.5877) ;
\draw(0,4.04584)  -- (0.3,4.04584)  ;
\draw(0.3,4.04584)  -- (0.3,4.34584)  ;

\draw (0,2.2877) -- (0,4.34584); 
\draw (1.18,1.96) arc (58.95034:90:2.2877)  ;
\draw (2.2415,3.7231) arc (58.95034:90:4.34584)  ;
\draw (2.2415,3.7231) arc (138.3705:159.53:5.6045)  ;

\draw (1.18,1.96) node[right] {\small{$e^{\mathbf{i}\varphi}$}}; 
\draw (2.2415,3.7231) node[right]{\small{$\lambda e^{\mathbf{i}\varphi}$}};
\draw (0,2.2877) node[left] {\small{$\mathbf{i}$}};
\draw (0,4.34584) node[left]{\small{$\lambda \mathbf{i}$}};
\draw (0.6,0.3) node {\small{$\varphi$}};
\draw (1,0) arc (0:58.95034:1)  ;
\end{tikzpicture}
	\caption{Saccheri quadrilateral in Corollary \ref{saco}.}
\end{figure}
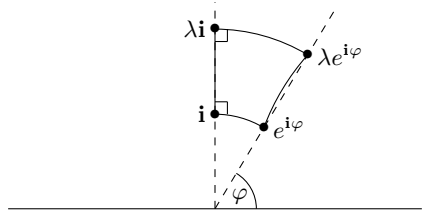

\section{Convex chains and distorted chains}\label{concha}
By a \emph{chain} we always mean an ordered set of points $x_0,x_1,\dots,x_n\in \Hy^2$. Such a chain \emph{lies} (or is \emph{contained}) in a curve $\gamma:I\rightarrow \Hy^2$ (where $I\subset \R$ is an interval) if $x_j=\gamma(t_j)$ and $t_0,t_1,t_2,\ldots \in I$ is a monotone sequence. 
In what follows we will only deal with hyperbolic geometry, so notions such as segments, polygons, half planes, convex hulls, etc. are always considered with respect to the hyperbolic distance.   
\begin{defi}\label{condef}
A chain $x_0,x_1,\dots,x_n\in \Hy^{2}$ is called \emph{convex} if the convex hull of $x_0,\dots,x_n$ is the polygon with sides $x_0x_1,x_1x_2,\dots,x_{n-1}x_n,x_nx_0$. 
\end{defi}
By a simple induction argument, every convex chain has non-negative tension.

\begin{defi} Let $x_0,\dots,x_n \in \Hy^{2}$ be a $\varphi$-good chain with $0<\varphi\leq \pi/2$. The $\varphi$-\emph{distorted chain} of $x_0,\dots,x_n$ is the chain $y_0,\dots,y_n$ in $\Hy^{2}$ contained in a curve with constant geodesic curvature equal to $\cos(\varphi)$ and so that $d(y_j,y_{j-1})=d(x_j,x_{j-1})$ for $1\leq j\leq n$.
\end{defi}


Clearly all distorted chains are convex. The goal of this section is to prove the following proposition, allowing us to work with convex chains contained in curves of constant geodesic curvature:
\begin{prop}\label{goodpos}
Let $x_0,\dots,x_n \in \Hy^{2}$ be a $\varphi$-good convex chain and consider the corresponding $\varphi$-distorted chain $y_0\dots,y_n$. 
Then 
\begin{equation*}
\tau(x_0,\dots,x_n)\leq \tau(y_0,\dots,y_n).
\end{equation*}
\end{prop}
We begin with a lemma.
\begin{lema}\label{comp3} Let $x_0,x_1,x_2$ be a $\varphi$-good chain in $\Hy^2$ with $0<\varphi\leq \pi/2$. If $y_0,y_1,y_2$ is the $\varphi$-distorted chain for $x_0,x_1,x_2$, then $\angle_{x_1}(x_0,x_2)\geq\angle_{y_1}(y_0,y_2)$.
\end{lema}
\begin{proof}
Assume that $y_0,y_1,y_2$ lie in the curve  $\mu:t\rightarrow te^{i\varphi}$ with $|y_2|>|y_0|$ and consider the $(a,b)$-canonical chain $z_0,z_1,z_2$ contained in $\mu$, with $z_1=y_1$ and $|z_2|>|z_0|$. The map $(x,y)\mapsto \frac{\sinh(x-y)}{\sinh(x)}$ is increasing in $x$ for $0\leq x$ and decreasing in $y$ for $0<y<x$, which by \eqref{lcos} implies
\begin{alignat*}{2}
\sin^{2}\left(\frac{\angle_{x_1}(x_0,x_2)}{2}\right) &=\frac{\sinh(d(x_0,x_1)-\groprod{x_0}{x_2}{x_1})\sinh(d(x_1,x_2)-\groprod{x_0}{x_2}{x_1})}{\sinh(d(x_0,x_1))\sinh(d(x_1,x_2))}\\
&\geq \frac{\sinh^{2}(a-b)}{\sinh^{2}(a)}=\sin^{2}\left(\frac{\angle_{z_1}(z_0,z_2)}{2}\right)
\end{alignat*}
and hence $\angle_{x_1}(x_0,x_2)\geq \angle_{z_1}(z_0,z_2)$. 

On the other hand, since $\min(d(y_0,y_1),d(y_1,y_2))\geq a=d(z_0,z_1)=d(z_1,z_2)$, for $j=0,2$, the point $z_{j}$ lies between the points $y_1$ and $y_j$ in $\mu$, implying $\angle_{z_1}(z_0,z_2)\geq\angle_{y_1}(y_0,y_2)$ and obtaining the desired inequality. \end{proof}
\begin{figure}[hbt]
\begin{tikzpicture}[scale=0.55]
\draw[dashed] (0,0) -- (15.1875,5.990625) ;
\draw (-1.5,0) -- (16,0);
\draw (3.51125,1.38499) node[below] {\small{$y_{0}$}}; 
\draw (3.51125,1.38499) node{\small{$\bullet$}};
\draw (5.4,2.13) node[below] {\small{$z_0$}};
\draw (5.4,2.13) node{\small{$\bullet$}};
\draw (8.1,3.195) node[below right] {\small{$y_{1}=z_1$}}; 
\draw (8.1,3.195) node{\small{$\bullet$}};
\draw (12.15,4.7925) node[below] {\small{$z_{2}$}}; 
\draw (12.15,4.7925) node{\small{$\bullet$}};
\draw (13.1562,5.189422) node[below] {\small{$y_{2}$}}; 
\draw (13.1562,5.189422) node{\small{$\bullet$}};
\draw (8.1,3.195) arc (66.471:156.581:3.4847)  ;
\draw (8.1,3.195) arc (84.639:138.413:3.20903)  ;
\draw (12.15,4.7925) arc (84.639:138.413:4.81355)  ;
\draw (13.1562,5.189422) arc (80.434:142.618:5.26259)  ;
\draw (12.15,4.7925) arc (67.249:155.803:5.1968)  ;
\draw (13.1562,5.189422) arc (55.806:167.246:6.2739)  ;
\draw (1.55,0.3) node {\small{$\varphi$}};
\draw (2,0) arc (0:21.5264:2)  ;
\end{tikzpicture}
	\caption{Proof of Lemma \ref{goodpos}.}
\end{figure}

Now we define the following process for a good convex chain $ \mathfrak{x}= x_0,\dots,x_n$. Fix $1\leq k\leq n-1$, and let $\alpha=\angle_{x_k}(x_{k-1},x_{k+1})$. For $\alpha\leq \gamma \leq \pi$, let $\mathfrak{x}^{(k)}(\gamma)$ be the unique convex chain  $x_0(\gamma),\dots,x_n(\gamma)$ satisfying $d(x_j(\gamma),x_{j-1}(\gamma))=d(x_j,x_{j-1})$ for $1\leq j\leq n$, $\angle_{x_j(\gamma)}(x_{j-1}(\gamma),x_{j+1}(\gamma))=\angle_{x_{j}}(x_{j-1},x_{j+1})$ for $1\leq j\leq n-1$ with $j\neq k$, and $\angle_{x_{k}(\gamma)}(x_{k-1}(\gamma),x_{k+1}(\gamma))=\gamma$. 


For such construction we prove the following:

\begin{lema}\label{lemachoro} Assume $d(x_0,x_n)\geq d(x_i,x_j)$ for all $0 \leq i<j \leq n$ with equality only if $i,j=0,n$. Then given $0\leq p\leq q\leq r\leq s\leq n$, the map $\gamma \mapsto t_{p,q,r,s}(\gamma)=d(x_{p}(\gamma),x_{s}(\gamma))-d(x_q(\gamma),x_r(\gamma))$ is non decreasing for $\alpha\leq \gamma \leq \pi$. In particular, $\gamma \mapsto \tau(\mathfrak{x}^{(k)}(\gamma))$ is non increasing. 
\end{lema}
\begin{proof}By an inductive argument it is enough to show the result for $p=0$, $r-q=n-1$ and $s=n$. Define $d_{i,j}=d(x_i(\gamma),x_j(\gamma))$ and $\alpha_{i,j,l}=\angle_{x_{j}(\gamma)}(x_i(\gamma),x_l(\gamma))$. We will compute the derivative $(t_{q,r})_{\gamma}=(t_{0,q,r,n})_{\gamma}=(d_{0,n})_{\gamma}-(d_{q,r})_{\gamma}$ presenting the computations for $q=0, r=n-1$ since the other case is similar. By \eqref{lcos} we have the relations
\begin{equation*}
\cosh(d_{0,n})=\cosh(d_{0,k})\cosh(d_{k,n})-\sinh(d_{0,k})\sinh(d_{k,n})\cos(\alpha_{0,k,n}),
\end{equation*}
\begin{equation*}
\cosh(d_{0,n-1})=\cosh(d_{0,k})\cosh(d_{k,n-1})-\sinh(d_{0,k})\sinh(d_{k,n-1})\cos(\alpha_{0,k,n-1}),
\end{equation*}
which by implicit differentiation imply
\begin{equation*}
(d_{0,n})_{\gamma}\sinh(d_{0,n})=\sinh(d_{0,k})\sinh(d_{k,n})\sin(\alpha_{0,k,n}),
\end{equation*}
\begin{equation*}
(d_{0,n-1})_{\gamma}\sinh(d_{0,n-1})=\sinh(d_{0,k})\sinh(d_{k,n-1})\sin(\alpha_{0,k,n-1}).
\end{equation*}
The law of sines \eqref{lsin} also gives us
\small{\begin{alignat*}{2}
(d_{0,n})_{\gamma}-(d_{0,n-1})_{\gamma}&=\sinh(d_{0,k})\left(\frac{\sinh(d_{k,n})\sin(\alpha_{k,0,n})}{\sinh(d_{0,n})}-\frac{\sinh(d_{k,n-1})\sin(\alpha_{k,0,n-1})}{\sinh(d_{0,n-1})}\right)\\
&=\sinh(d_{0,k})(\sin(\alpha_{k,0,n})-\sin(\alpha_{k,0,n-1})).
\end{alignat*}}
\normalsize
Since $x_0(\gamma),\dots,x_n(\gamma)$ is a convex chain for all $\alpha\leq \gamma \leq \pi$, we have $0\leq \alpha_{k,0,n-1}\leq \alpha_{k,0,n}$, and hence $(t_{0,n-1})_{\gamma}\geq 0$ whenever $\alpha_{k,0,n}\leq \pi/2$. Similarly, $(t_{1,n})_{\gamma}\geq 0$ whenever $\alpha_{0,n,k}\leq \pi/2$. To prove that both conditions hold for $\alpha\leq \gamma \leq \pi$, let $\alpha\leq u\leq \pi$ be the maximal angle so that $\alpha_{k,0,n},\alpha_{0,n,k}\leq \pi/2$ and  $(t_{0,n-1})_{\gamma},(t_{1,n})_{\gamma}\geq 0$ for all $\alpha\leq \gamma<u$. We will prove that $u=\pi$. 

Our assumption about $x_0,\dots, x_n$ implies $$d(x_0(\gamma),x_n(\gamma))>d(x_0(\gamma),x_{n-1}(\gamma)),d(x_1(\gamma),x_{n}(\gamma)),$$ and hence $\angle_{x_0(\gamma)}(x_k(\gamma),x_n(\gamma)),\angle_{x_n(\gamma)}(x_0(\gamma),x_k(\gamma))<\pi/2$ for $\gamma$ in a neighborhood of $\alpha$, implying $\alpha<u$. But if $u<\pi$, since $(t_{1,n})_{\gamma}$ is non negative in $(\alpha,u)$, by the mean value theorem we have $$d(x_0(u),x_n(u))-d(x_0(u),x_{n-1}(u))\geq d(x_0(\alpha),x_n(\alpha))-d(x_0(\alpha),x_{n-1}(\alpha))>0,$$ implying $\angle_{x_0(\gamma)}(x_k(\gamma),x_n(\gamma))<\pi/2$, in a neighborhood of $u$. This also happens for $(t_{0,n-1})_{\gamma}$, contradicting the definition of $u$ and completing the proof of the lemma.
\end{proof}

\begin{proof}[Proof of Proposition \ref{goodpos}]
Let $\mathfrak{x}=x_0,\dots,x_n$, $\mathfrak{y}=y_0,\dots,y_n$, and consider the sequence $\mathfrak{z}_{0}, \mathfrak{z}_{1},\dots,\mathfrak{z}_{n-1} $ of convex chains defined inductively by $\mathfrak{z}_{0}=\mathfrak{y}$, and $\mathfrak{z}_{k}=\mathfrak{z}_{k-1}^{(k)}(\angle_{x_k}(x_{k-1},x_{k+1}))$ for $1\leq k\leq n-1$. By Lemma \ref{comp3}, $\angle_{x_{k}}(x_{k-1},x_{k+1})\geq \angle_{y_k}(y_{k-1},y_{k+1})$ for $1\leq k\leq n-1$, and since $\mathfrak{z}_0$ lies in a curve of constant geodesic curvature less than 1, we are in the assumptions of Lemma \ref{lemachoro}, therefore $\tau(\mathfrak{z}_{1}) \leq \tau(\mathfrak{z}_{0})$. Inductively, this assumption holds for every $1 \leq k\leq n-1$, and hence $\tau(\mathfrak{z}_{1}) \leq \tau(\mathfrak{z}_{0})$. Since  $\mathfrak{z}_{n-1}=\mathfrak{x}$, we are done.
\end{proof}
As an immediate corollary of the previous proofs we have:
\begin{coro}\label{agu}
If $x_0,x_1,\dots,x_n$ is a good convex chain in $\Hy^2$ then $\angle_{x_0}(x_1,x_n)\leq\pi/2$ and $d(x_0,x_n)\geq d(x_i,x_j)$ for all $0\leq i<j\leq n$.
\end{coro}

\begin{coro}\label{congood}
Every sub-chain of a convex good chain is good.
\end{coro}
\begin{proof}
For $x_0,\dots,x_n$ convex and $\varphi$-good, it is enough to show that $x_0,x_k,x_n$ is $\varphi$-good for arbitrary $1\leq k\leq n-1$.  If $y_0,\dots,y_n$ is the $\varphi$-distorted chain for $x_0,\dots, x_n$, Lemmas \ref{comp3} and \ref{lemachoro} imply $\angle_{x_k}(x_0,x_n)\geq \angle_{y_k}(y_0,y_n)$, $d(y_0,y_k) \leq d(x_0,x_k)$ and $d(y_k,y_n)\leq d(x_k,x_n)$. This means that the chain $z_0,z_k,z_n$ contained in a curve with constant geodesic curvature $\cos(\varphi)$ and with $d(z_k,z_j)=d(x_k,x_j)$ for $j=0,n$ satisfies $\angle_{z_k}(z_0,z_n)\leq \angle_{y_k}(y_0,y_n) \leq \angle_{x_k}(x_0,x_n)$ and hence $x_0,x_k,x_n$ is $\varphi$-good by Lemma \ref{comp3}.
\end{proof}

\section{Proof of the Avalanche Principle: convex case}\label{APCC}
For $\lambda_1,\lambda_2,\dots,\lambda_n>1$ and $0<\varphi\leq \pi/2$, consider the chain $\mathfrak{x}(\lambda_1,\dots,\lambda_n;\varphi)=x_0,x_1,\dots,x_n$ given by $x_0=e^{\mathbf{i}\varphi}$, and $x_j=\lambda_1\cdots\lambda_j\cdot x_0$ for $j\geq 1$. Every $\varphi$-distorted chain can be considered of this form.
\begin{prop}\label{varcte} The function $\varphi \mapsto \tau(\mathfrak{x}(\lambda_1,\dots,\lambda_n;\varphi))$ is non-increasing for $0 <\varphi\leq \pi/2$.
\end{prop}
We need a lemma:
\begin{lema}\label{lematanh} If $x,y,z,w \geq0 $ satisfy $\min(x,y,z,w)=x$, $\max(x,y,z,w)=y$, and $x+y\leq z+w$, then:
\begin{equation*}
\tanh(x)+\tanh(y)\leq \tanh(z)+\tanh(w).
\end{equation*}
\end{lema}
\begin{proof} The $\tanh$ function is increasing, so it is enough to prove the result replacing  $x$ by $x'=z+w-y$ under the assumptions $z \leq w \leq y$ and $x'\leq y$. In this case we have
\begin{alignat*}{2}
\cosh(z)\cosh(w) & =2[\cosh(z+w)+\cosh(w-z)]\\ & \leq 2[\cosh(z+w)+\cosh(w-z+2(y-w))] \\ 
&= 2[\cosh(x'+y)+\cosh(y-x')] \\ 
&=\cosh(x')\cosh(y), \\
\end{alignat*}
and we conclude
\begin{equation*}
\tanh(x')+\tanh(y)=\frac{\sinh(x'+y)}{\cosh(x')\cosh(y)}\leq \frac{\sinh(z+w)}{\cosh(z)\cosh(w)}=\tanh(z)+\tanh(w). \qedhere
\end{equation*}
\end{proof}
\begin{proof}[Proof of Proposition \ref{varcte}]
Since $$\tau(\mathfrak{x}(\lambda_1,\dots,\lambda_n;\varphi))=\tau(\mathfrak{x}(\lambda_1,\dots,\lambda_{n-1};\varphi))+\tau(\mathfrak{x}(\lambda_1\cdots\lambda_{n-2},\lambda_{n-1},\lambda_n;\varphi)),$$ it is enough to show that $g(t)=\tau(\mathfrak{x}(a,b,c;\textup{arccsc}(t)))$ is non-decreasing for $t\geq 1$, which is the same as $g'(t)\geq 0$.

For $0\leq i<j\leq 3$, let $d_{i,j}=d(x_i,x_j)$ and $s_{i,j}=d(\lambda_{1}\cdots \lambda_i,\lambda_{1}\cdots \lambda_j)$. By Corollary \ref{saco} we have $\sinh(d_{i,j}/2)=t\cdot\sinh(s_{i,j}/2)$ and hence $(d_{i,j})':=\frac{\partial d_{i,j}}{\partial t}=2\frac{\sinh(s_{i,j}/2)}{\cosh(d_{i,j}/2)}=2t^{-1}\tanh(d_{i,j}/2)$. We obtain 
$g'(t)=(d_{0,2})'+(d_{1,3})'-(d_{1,2})'-(d_{0,3})'=2t^{-1}(\tanh(d_{0,2}/2)+\tanh(d_{1,3}/2)-\tanh(d_{1,2}/2)-\tanh(d_{0,3}/2))$, and since $d_{0,2}+d_{1,3}\geq d_{1,2}+d_{0,3}$ and $d_{0,3} \geq d_{0,2},d_{1,3} \geq d_{1,2}$, Lemma \ref{lematanh} applies, concluding $g'(t)\geq 0$.
\end{proof}
\begin{proof}[Proof of Theorem \ref{AP} (convex case)]
Let $x_0,\dots,x_n \in \Hy^2$ be a convex $(a,b)$-good pair. By Proposition \ref{goodpos} we may assume $x_0,\dots,x_n=\mathfrak{x}(\lambda_1,\dots,\lambda_n;\varphi)$ where $\varphi$ is the curvature angle of $(a,b)$ and
\begin{equation}
    \sinh(\log(\lambda_i)/2)=\sin(\varphi)\sinh(d(x_i,x_{i-1})/2). \label{lambdas}
\end{equation}
Also, by Proposition \ref{varcte} we have 
\begin{equation*}
\tau(x_0,\dots,x_n)\leq \tau(\mathfrak{x}(\lambda_1,\dots,\lambda_n;0)):=\displaystyle\lim_{\alpha\to 0^{+}}{\tau(\mathfrak{x}(\lambda_1,\dots,\lambda_n;\alpha))}.
\end{equation*}If we define $g_{\alpha}(x,y)=|y-x|+\sqrt{(y-x)^2+4xy\sin(\alpha)}$, $t_0=1$ and $t_j=\lambda_1\cdots \lambda_j$ for $1\leq j \leq n$, by using \eqref{disthip} we obtain the formula 
\small{\begin{align*}
    \tau(\mathfrak{x}(\lambda_1,\dots,\lambda_n;\alpha)) & =2\log{\left(\frac{\sqrt{t_0t_n}\sin(\alpha)}{g_{\alpha}(t_0,t_n)}\left(\prod_{j=0}^{n-2}{\frac{g_{\alpha}(t_j,t_{j+2})}{\sqrt{t_jt_{j+2}}\sin(\alpha)}}\right) \left(\prod_{j=1}^{n-2}{\frac{\sqrt{t_jt_{j+1}}\sin(\alpha)}{g_{\alpha}(t_j,t_{j+1})}}\right) \right)} \\
    & = 2\log{\left(\frac{\prod_{j=0}^{n-2}{g_{\alpha}(t_j,t_{j+2})}}{g_{\alpha}(t_0,t_n)\prod_{j=1}^{n-2}{g_{\alpha}(t_j,t_{j+1})}}\right)}
\end{align*}}
\normalsize
Also, the identity $g_{0}(x,y)=2|y-x|$ implies
\begin{align*}
    \tau(\mathfrak{x}(\lambda_1,\dots,\lambda_n;0)) & = 2\log{\left(\frac{(\lambda_1\lambda_2-1)(\lambda_2\lambda_3-1)(\lambda_3\lambda_4-1)\cdots (\lambda_{n-1}\lambda_{n}-1)}{(\lambda_1\cdots\lambda_n-1)(\lambda_2-1)(\lambda_3-1)\cdots(\lambda_{n-1}-1)}\right)}.
\end{align*}
Since $\lambda_{j}>1$ for each $j$, we also have the inequality 
\small{\begin{align*}
    (\lambda_1\lambda_2-1)(\lambda_2\lambda_3-1)\cdots(\lambda_{n-1}\lambda_n-1) & =\frac{(\lambda_1\cdots\lambda_n-\lambda_3\cdots\lambda_n)(\lambda_2\lambda_3-1)\cdots(\lambda_{n-1}\lambda_n-1)}{(\lambda_3\cdots\lambda_{n-1}\lambda_n)} \\
    & \leq \frac{(\lambda_1\cdots\lambda_n-1)(\lambda_2\lambda_3)\cdots(\lambda_{n-1}\lambda_n)}{(\lambda_3\cdots\lambda_{n-1}\lambda_n)} \\
    & =(\lambda_1\cdots\lambda_n-1)(\lambda_2\lambda_3\cdots \lambda_{n-1})\\
\end{align*}}
\normalsize implying
\begin{align*}
    \tau(\mathfrak{x}(\lambda_1,\dots,\lambda_n;0))&  \leq 2\log{\left(\frac{\lambda_2}{(\lambda_2-1)}\cdot \frac{\lambda_3}{(\lambda_3-1)}\cdots \frac{\lambda_{n-1}}{(\lambda_{n-1}-1)}\right)}\\
    & =2\sum_{j=2}^{n-1}{\log{\left(\frac{1}{\lambda_j-1}+1 \right)}} \leq 2\sum_{j=2}^{n-1}{\frac{1}{\lambda_j-1}}.
    \end{align*}
It only remains to note that $\lambda_j \geq \lambda$ for $2\leq j\leq n-1$, which follows from \eqref{parameters} and \eqref{lambdas} since
\begin{equation*}
\sinh(\log(\lambda_i)/2)  =\sin(\varphi)\sinh(d(x_i,x_{i-1})/2)
  \geq \sin(\varphi)\sinh(a/2)  = \sinh(\log(\lambda)/2).
\end{equation*}
The proof is complete in this case.
\end{proof}
\section{Avalanche Principle for $\textup{CAT}(-1)$ spaces}\label{sAPCAT}
In this section we define $\CAT{-1}$ spaces and prove Theorem \ref{catcom}. For three distinct points $x,y,z$ in a geodesic metric space $X$, a geodesic triangle with vertices $x,y,z$ will be denoted by $\triangle(x,y,z)$. For such a triangle, a \emph{comparison triangle} will be a geodesic triangle $\overline{\triangle(x,y,z)}=\triangle(\overline{x},\overline{y},\overline{z})\subset \Hy^{2}$, with $d(p,q)=d(\overline{p},\overline{q})$ for $p,q=x,y,z$. If $p$ belongs to a side of $\triangle(x,y,z)$, say at $xy$, the \emph{comparison} point of $p$ is the unique point $\overline{p}$ in the side $\overline{x}\overline{y}$ of $\overline{\triangle(x,y,z)}$ satisfying $d(p,x)=d(\overline{p},\overline{x})$. 

\begin{defi}The metric space $X$ is a $\textup{CAT}(-1)$ space if it is geodesic, and for every geodesic triangle $\triangle(x,y,z)$ in $X$ and every pair of points $p,q$ in sides of $\triangle(x,y,z)$, the corresponding comparison points $\overline{p},\overline{q}$ in $\triangle(\overline{x},\overline{y},\overline{z})\subset \Hy^{2}$ satisfy $d(p,q)\leq d(\overline{p},\overline{q})$.
\end{defi}

We have a characterization of $\textup{CAT}(-1)$ spaces in terms of the law of cosines. For $a,b>0$ and $0\leq\gamma\leq\pi$, let $$L(a,b,\gamma)=\arccosh(\cosh(a)\cosh(b)-\sinh(a)\sinh(b)\cos(\gamma)).$$ The law of cosines implies $c=L(a,b,\gamma)$ if and only if a geodesic triangle in $\Hy^{2}$ with sides $a,b,c$ has the property that the angle corresponding to $c$ equals $\gamma$. Clearly $L$ is symmetric in the first two variables and increasing in the third variable, and the identity
\begin{equation*}
    \cosh(L(a,b,\gamma))=\cosh(a-b)+2\sinh(a)\sinh(b)\sin^2(\gamma/2)
\end{equation*} implies that for $b$ and $\gamma$ fixed, $L$ is increasing in the first variable while $a\geq b$ or $\gamma \geq \pi/2$. With this notation, a metric space $X$ is $\textup{CAT}(-1)$ if it is geodesic, and for every geodesic triangle $\triangle(x,y,z)$ with $a=d(y,z), b=d(z,x),c=(x,y)$ and $\angle_{z}(x,y)=\gamma$, we have
\begin{equation*}
c\geq L(a,b,\gamma),
\end{equation*}
where $\angle_{z}(x,y)$ denotes Alexandrov angle (see \cite[Ch.~II, Prop.~1.7]{bri}).

\begin{defi}\label{compdef} For a chain $x_0,\dots,x_n$ in a metric space $X$, its \emph{comparison chain} is the essentially unique convex chain $\ov{x}_0,\dots,\ov{x}_n\in \Hy^{2}$ so that  $\ov{\triangle(x_{j-1},x_{j},x_{j+1})}=\triangle(\ov{x}_{j-1},\ov{x}_j,\ov{x}_{j+1})$ for $1\leq j\leq n-1$.  
\end{defi}

We begin the proof of Theorem \ref{catcom} with a lemma relating convex and non-convex chains in $\Hy^{2}$.
\begin{lema}\label{crece}Let $x_0,x_1,x_2$ be a good chain in $\Hy^2$, with $a=d(x_0,x_1)$, $b=d(x_1,x_2)$, and fix $e>b$. For $0<\gamma<\pi$, let $x_{3}(\gamma)$ the unique point in the same half-plane of $x_0$ w.r.t.~the geodesic determined by $x_1x_2$, such that $\angle_{x_1}(x_2,x_3(\gamma))=\gamma$ and $d(x_1,x_3(\gamma))=e$. Let $y_{3}(\gamma)$ be the reflection of $x_{3}(\gamma)$ with respect to the geodesic containing $x_1x_2$, and let $0<u<\pi$ be such that $x_0,x_1,x_2,x_3(\gamma)$ is a good convex chain for all $0<\gamma<u$. Then the map $\gamma \mapsto \tau(x_0,x_1,x_2,x_3(\gamma))+\tau(x_0,x_1,x_2,y_3(\gamma))$ is non-decreasing for $0<\gamma<u$.
\end{lema}

\begin{figure}[hbt]
\begin{tikzpicture}[scale=0.9]
\draw (-2,0) -- (2,0);
\draw (2,0) -- (-3.66,2.03);
\draw (-2,0) -- (3.64,1.33);
\draw (2,0) -- (3.64,-1.33);
\draw (2,0) -- (3.64,1.33);
\draw (-2,0) -- (-3.66,2.03);
\draw[dashed] (-3.66,2.03) -- (3.64,1.33);
\draw[dashed] (-3.66,2.03) -- (3.64,-1.33);
\draw (-3.66,2.03) node[left] {\small{$x_{0}$}}; 
\draw (-3.66,2.03) node{\small{$\bullet$}};
\draw (-2,0) node[left] {\small{$x_{1}$}};  
\draw (-2,0) node{\small{$\bullet$}};
\draw (2,0) node[right] {\small{$x_{2}$}}; 
\draw (2,0) node{\small{$\bullet$}};
\draw (3.64,1.33) node[right] {\small{$x_{3}(\gamma)$}}; 
\draw (3.64,1.33) node{\small{$\bullet$}};
\draw (3.64,-1.33) node[right] {\small{$y_{3}(\gamma)$}};
\draw (3.64,-1.33) node{\small{$\bullet$}};
\draw (0.3,0) arc (0:13.26880:2.3)  ;
\draw[dashed] (3.64,-1.33) arc (-13.26880:13.26880:5.79469)  ;
\draw[dashed] (-1.1,0) arc (0:128.26880:0.9)  ;
\draw (-1.75,0.17) node[above] {\small{$\beta$}};
\draw (-0.3,0.19) node {\small{$\gamma$}};
\draw (-2.83,1.01) node[left] {\small{$a$}};
\draw (0,0) node[below] {\small{$b$}};
\draw (2.82,0.665) node[left] {\small{$c$}};
\draw (0.82,0.665) node[above] {\small{$e$}};
\draw (-0.01,1.63) node[above] {\small{$f$}};
\draw (1.6428,-0.38683) node[below] {\small{$g$}};
\end{tikzpicture}
\caption{Proof of Lemma \ref{crece}.}
\end{figure}
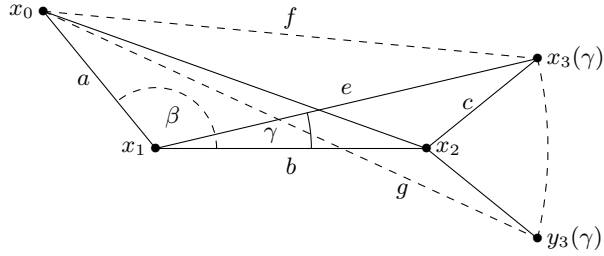
\begin{proof}Let $\beta=\angle_{x_1}(x_0,x_2)$, $c=c(\gamma)=d(x_2,x_3(\gamma))=d(x_2,y_3(\gamma))$, $f=f(\gamma)=d(x_0,x_3(\gamma))$, and $g=g(\gamma)=d(x_0,y_3(\gamma))$. It is enough to show that $f_{\gamma}+g_{\gamma}\leq 0$ for $0<\gamma<u$. We will use the following relations coming from \eqref{lcos}:
\begin{equation*}
\cosh(f)=\cosh(a)\cosh(e)-\sinh(a)\sinh(e)\cos(\beta-\gamma),
\end{equation*}
\begin{equation*}
\cosh(g)=\cosh(a)\cosh(e)-\sinh(a)\sinh(e)\cos(\beta+\gamma).
\end{equation*}
Implicit differentiation gives us
\begin{equation*}
f_{\gamma}\sinh(f)=-\sinh(a)\sinh(e)\sin(\beta-\gamma),
\end{equation*}
\begin{equation*}
g_{\gamma}\sinh(g)=\sinh(a)\sinh(e)\sin(\beta+\gamma),
\end{equation*}
and by \eqref{lsin} applied to the triangles $x_0,x_1,y_3(\gamma)$ and $x_0,x_1,x_3(\gamma)$ respectively, we obtain
\begin{alignat*}{2}
f_{\gamma}+g_{\gamma} & =\sinh(a)\sinh(e)\left(\frac{\sin(\beta+\gamma)}{\sinh(g)}-\frac{\sin(\beta-\gamma)}{\sinh(f)}\right)\\
& = \sinh(a)(\sin(\angle_{x_0}(x_1,y_3(\gamma)))-\sin(\angle_{x_0}(x_1,x_3(\gamma)))).
\end{alignat*}
Since $x_0,x_1,x_2,x_3(\gamma)$ is a good convex chain, by Corollary \ref{agu} we obtain $0\leq \angle_{x_0}(x_1,y_3(\gamma))\leq \angle_{x_0}(x_1,x_3(\gamma)) \leq \pi/2$, and hence $f_{\gamma}+g_{\gamma}\leq 0$, as desired.
\end{proof}
\begin{coro}\label{cororef}Let $x_0,x_1,x_2,x_3$ be a good convex chain in $\Hy^{2}$, and let $y_3$ be the reflection of $x_3$ with respect to the segment $x_1x_2$. Then
\begin{equation}
\tau(x_0,x_1,x_2,x_3)+\tau(x_0,x_1,x_2,y_3)\geq 0.
\end{equation}
\end{coro}
\begin{proof}
Consider $x_3=x_3(\gamma)$ as a variable point depending on $\gamma=\angle_{x_1}(x_2,x_3)$, as in the statement of Lemma \ref{crece}. By this lemma we obtain
\begin{equation*}
 \tau(x_0,x_1,x_2,x_3)+\tau(x_0,x_1,x_2,y_3)\geq \tau(x_0,x_1,x_2,x_3(0))+\tau(x_0,x_1,x_2,y_3(0)).
\end{equation*}
At $\gamma=0$ we have $x_3(0)=y_3(0)$, and $d(x_1,x_3(0))=d(x_1,x_2)+d(x_2,x_3(0))$. Then
\begin{alignat*}{2}
 & \hspace{4mm} \tau(x_0,x_1,x_2,x_3(0))+\tau(x_0,x_1,x_2,y_3(0)) \\  &=2\tau(x_0,x_1,x_2,x_3(0)) \\
&=2(d(x_0,x_2)+d(x_1,x_3(0))-d(x_1,x_2)-d(x_0,x_3(0)) \\
&=2(d(x_0,x_2)+d(x_2,x_{3}(0))-d(x_0,x_3(0)))\geq 0.
\end{alignat*}
The conclusion follows.
\end{proof}
The main ingredient of the proof of Theorem \ref{catcom} is the case for $n=3$, which we prove now:
\begin{prop}\label{cato}Suppose that $X$ is a $\textup{CAT}(-1)$ space, and consider a good chain $x_0,x_1,x_2,x_3$ in $X$ with respective comparison chain $\overline{x}_0,\overline{x}_1,\ov{x}_2,\overline{x}_3$ in $\Hy^{2}$. Then 
\begin{equation*}
|\tau(x_0,x_1,x_2,x_3)|\leq \tau(\overline{x}_0,\overline{x}_1,\ov{x}_2,\overline{x}_3).
\end{equation*}
\end{prop}
\begin{proof}
The inequality  $\tau(x_0,x_1,x_2,x_3)\leq \tau(\overline{x}_0,\overline{x}_1,\overline{x}_2,\overline{x}_3)$ turns out to be equivalent to $d(x_0,x_3)\geq d(\overline{x}_0,\overline{x}_3)$, so we will prove it first. Let $\ov{P}=\ov{x}_0\ov{x}_2\cap\ov{x}_1\ov{x}_3$, and consider the point $P\in x_0x_2$ such that $\ov{P}$ is the comparison point for $P$ in $\triangle(x_0,x_1,x_2)$. The $\textup{CAT}(-1)$ inequality implies $d(P,x_1)\leq d(\ov{P},\ov{x}_1)$, and hence $d(P,x_3)\geq d(x_1,x_3)-d(P,x_1)\geq d(\ov{x}_1,\ov{x}_3)-d(\ov{P},\ov{x}_1)=d(\ov{P},\ov{x}_3)$. In addition, since $\overline{x}_0,\overline{x}_1,\overline{x}_2,\overline{x}_3$ is a convex good chain in $\Hy^2$, Corollary \ref{agu} implies
\begin{equation*}
    \frac{\sinh(d(\ov{P},\ov{x}_3))}{\sinh((d(\ov{P},\ov{x}_2))}=\frac{\sin(\angle_{\ov{x}_2}(\ov{x}_0,\ov{x}_3))}{\sin(\angle_{\ov{x}_3}(\ov{x}_1,\ov{x}_2))}\geq \frac{\sin(\angle_{\ov{x}_2}(\ov{x}_0,\ov{x}_3))}{\sin(\angle_{\ov{x}_3}(\ov{x}_0,\ov{x}_2))}=\frac{\sinh((d(\ov{x}_0,\ov{x}_3))}{\sinh((d(\ov{x}_0,\ov{x}_2))}\geq 1,
\end{equation*} and we have $d(P,x_3) \geq d(\ov{P},\ov{x}_3) \geq d(\ov{P},\ov{x}_2) $. The monotonicity properties of $L$ imply 
\begin{alignat*}{2}
L(d(\ov{P},\ov{x}_2),d(\ov{P},\ov{x}_3),\angle_{\ov{P}}(\ov{x}_2,\ov{x}_3)) &=d(\ov{x}_2,\ov{x}_3)=d(x_2,x_3)\\
&\geq L(d(P,x_2),d(P,x_3),\angle_{P}(x_2,x_3))\\
&\geq L(d(\ov{P},\ov{x}_2),d(\ov{P},\ov{x}_3),\angle_{P}(x_2,x_3)),
\end{alignat*}
so $\angle_{\ov{P}}(\ov{x}_2,\ov{x}_3)\geq \angle_{P}(x_2,x_3)$, and since $\angle_{\ov{P}}(\ov{x}_2,\ov{x}_3)\leq \pi/2$ we obtain $\angle_{P}(x_0,x_3)\geq \angle_{P}(x_0,x_2)-\angle_{P}(x_2,x_3)\geq \angle_{\ov{P}}(\ov{x}_0,\ov{x}_2)-\angle_{\ov{P}}(\ov{x}_2,\ov{x}_3)=\angle_{\ov{P}}(\ov{x}_0,\ov{x}_3) \geq \pi/2$.
Therefore
\begin{alignat*}{2}
d(x_0,x_3)&\geq L(d(P,x_0),d(P,x_3),\angle_{P}(x_0,x_3))\\
&\geq L(d(\ov{P},\ov{x}_0),d(\ov{P},\ov{x}_3),\angle_{P}(x_0,x_3))\\
&\geq L(d(\ov{P},\ov{x}_0),d(\ov{P},\ov{x}_3),\angle_{\ov{P}}(\ov{x}_0,\ov{x}_3))=d(\ov{x}_0,\ov{x}_3).
\end{alignat*}
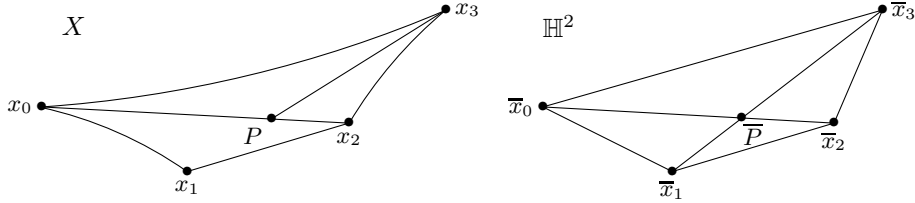
\begin{figure}[hbt]
\begin{tikzpicture}[scale=0.43]
\draw (2.5,7.5) node {$\Hy^2$};
\draw (2,5) -- (6,3);
\draw (2,5) -- (11,4.5);
\draw (2,5) -- (12.5,8);
\draw (6,3) -- (11,4.5);
\draw (6,3) -- (12.5,8);
\draw (11,4.5) -- (12.5,8);
\draw (2,5) node[left] {\small{$\overline{x}_{0}$}};
\draw (2,5) node{\small{$\bullet$}};
\draw (6,3) node[below] {\small{$\overline{x}_{1}$}};
\draw (6,3) node{\small{$\bullet$}};
\draw (11,4.5) node[below]{\small{$\overline{x}_{2}$}};
\draw (11,4.5) node{\small{$\bullet$}};
\draw (12.5,8) node[right] {\small{$\overline{x}_{3}$}};
\draw (12.5,8) node{\small{$\bullet$}};
\draw (7.905047,4.783085) node[below right] {\small{$\overline{P}$}};
\draw (8.145047,4.683085) node{\small{$\bullet$}};
\draw (-12.5,7.5) node {$X$};
\draw (-13.5,5) -- (-4,4.5);
\draw (-6.387062,4.633085) -- (-1,8);
\draw (-9,3) -- (-4,4.5);
\draw (-9,3) arc (54.91717434:77.69820359:12.5380) ;
\draw (-1,8) arc (129.86006949:149.19324577:13.7394) ;
\draw (-13.5,5) arc (-86.14661272:-66.84809096:38.454849);
\draw (-13.5,5) node[left] {\small{$x_{0}$}};
\draw (-13.5,5) node{\small{$\bullet$}};
\draw (-9,3) node[below] {\small{$x_{1}$}};
\draw (-9,3) node{\small{$\bullet$}};
\draw (-4,4.5) node[below]{\small{$x_{2}$}};
\draw (-4,4.5) node{\small{$\bullet$}};
\draw (-1,8) node[right] {\small{$x_{3}$}};
\draw (-1,8) node{\small{$\bullet$}};
\draw (-6.387062,4.633085) node[below left] {\small{$P$}};
\draw (-6.387062,4.633085) node{\small{$\bullet$}};
\end{tikzpicture}
\caption{Proof of $\tau(x_0,x_1,x_2,x_3)\leq \tau(\overline{x}_0,\overline{x}_1,\overline{x}_2,\overline{x}_3)$.}
\end{figure}

The second inequality is $\tau(x_0,x_1,x_2,x_3)\geq -\tau(\overline{x}_0,\overline{x}_1,\overline{x}_2,\overline{x}_3)$. For this one, let $\ov{y}_3$ be the reflection of $\ov{x}_3$ with respect to $\ov{x}_1\ov{x}_2$. We separate into two cases:

Case 1: The chain $\ov{x}_0,\ov{x}_1,\ov{y}_3,\ov{x}_2$ is convex.

Let $\ov{P}=\ov{x}_1\ov{x}_2\cap\ov{x}_0\ov{y}_3$ and consider the point $P\in x_1x_2$ so that $\ov{P}$ is the comparison point of $P$ in $\ov{x}_1\ov{x}_2$. The $\textup{CAT}(-1)$ inequality applied to $\triangle(x_0,x_1,x_2)$ implies $d(P,x_0)\leq d(\ov{P},\ov{x}_0)$, and similarly $d(P,x_3)\leq d(\ov{P},\ov{x}_3)$, obtaining $d(x_0,x_3)\leq d(x_0,P)+d(P,x_3)\leq d(\ov{P},\ov{x}_0)+d(\ov{P},\ov{x}_3)=d(\ov{x}_0,\ov{y}_3)$. This last inequality is equivalent to $\tau(x_0,x_1,x_2,x_3)\geq \tau(\ov{x}_0,\ov{x}_1,\ov{x}_2,\ov{y}_3)$ and by Corollary \ref{cororef} we obtain 
$\tau(x_0,x_1,x_2,x_3)\geq-\tau(\ov{x}_0,\ov{x}_1,\ov{x}_2,\ov{x}_3)$.

\begin{figure}[hbt]
\begin{tikzpicture}[scale=0.4]
\draw (2.5,7.5) node {$\Hy^2$};
\draw (2,5) -- (6,3);
\draw (2,5) -- (11,4.5);
\draw (2,5) -- (12.5,8);
\draw (6,3) -- (11,4.5);
\draw (6,3) -- (14.12408,2.266972);
\draw (11,4.5) -- (14.12408,2.266972);
\draw (2,5) -- (14.12408,2.266972);
\draw (11,4.5) -- (12.5,8);
\draw (2,5) node[left] {\small{$\overline{x}_{0}$}};
\draw ((2,5) node{\small{$\bullet$}};
\draw (6,3) node[below] {\small{$\overline{x}_{1}$}};
\draw (6,3) node{\small{$\bullet$}};
\draw (11,4.5) node[below]{\small{$\overline{x}_{2}$}};
\draw (11,4.5) node{\small{$\bullet$}};
\draw (12.5,8) node[right] {\small{$\overline{x}_{3}$}};
\draw (12.5,8) node{\small{$\bullet$}};
\draw (14.12408,2.266972) node[right] {\small{$\overline{y}_{3}$}};
\draw (14.12408,2.266972) node{\small{$\bullet$}};
\draw (7.905047,4.783085) node[below right] {\small{$\overline{P}$}};
\draw (8.115047,3.633085) node{\small{$\bullet$}};
\draw (-12.5,7.5) node {$X$};
\draw (-13.5,5) -- (-6.91526,3.620630);
\draw (-1,8) -- (-6.91526,3.620630);
\draw (-13.5,5) -- (-4,4.5);
\draw (-9,3) -- (-4,4.5);
\draw (-9,3) arc (54.91717434:77.69820359:12.5380) ;
\draw (-1,8) arc (129.86006949:149.19324577:13.7394) ;
\draw (-13.5,5) arc (-86.14661272:-66.84809096:38.454849);
\draw (-13.5,5) node[left] {\small{$x_{0}$}};
\draw (-13.5,5) node{\small{$\bullet$}};
\draw (-9,3) node[below] {\small{$x_{1}$}};
\draw (-9,3) node{\small{$\bullet$}};
\draw (-4,4.5) node[below]{\small{$x_{2}$}};
\draw (-4,4.5) node{\small{$\bullet$}};
\draw (-1,8) node[right] {\small{$x_{3}$}};
\draw (-1,8) node{\small{$\bullet$}};
\draw (-6.387062,4.633085) node[below left] {\small{$P$}};
\draw (-6.927062,3.633085) node{\small{$\bullet$}};
\end{tikzpicture}
\caption{Proof of $\tau(x_0,x_1,x_2,x_3)\geq -\tau(\overline{x}_0,\overline{x}_1,\overline{x}_2,\overline{x}_3)$: Case 1.}
\end{figure}
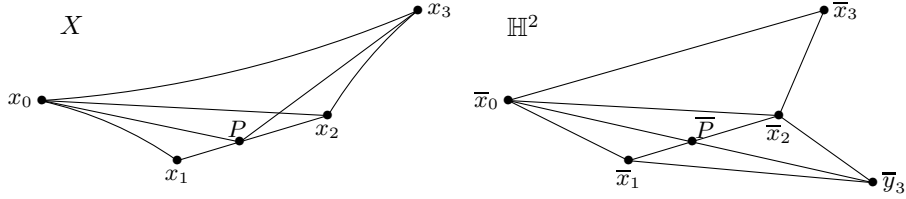
Case 2: The chain $\ov{x}_0,\ov{x}_1,\ov{y}_3,\ov{x}_2$ is not convex.

W.l.o.g. suppose that $\ov{x}_2$ is an interior point of convex hull of $\ov{x}_0,\ov{x}_1,\ov{y}_3$, and consider $\ov{x}_3$ and $\ov{y}_3$ as points depending on $\gamma=\angle_{\ov{x}_1}(\ov{x}_2,\ov{x}_3)$ as in Lemma \ref{crece}. Let $\beta=\angle_{\ov{x}_1}(\ov{x}_0,\ov{x}_2)$, $a=d(\ov{x}_0,\ov{x}_1)$, $b=d(\ov{x}_1,\ov{x}_2)$, $c=c(\gamma)=d(\ov{x}_2,\ov{x}_3(\gamma))=d(\ov{x}_2,\ov{y}_3(\gamma))$, $d=d(\ov{x}_0,\ov{x}_2)$, $e=d(\ov{x}_1,\ov{x}_3(\gamma))=d(\ov{x}_1,\ov{y}_3(\gamma))$, $f=f(\gamma)=d(\ov{x}_0,\ov{x}_3(\gamma))$, and $g=g(\gamma)=d(\ov{x}_0,\ov{y}_3(\gamma))$.
Also, let $0<u<\pi$ be the angle so that $\angle_{\ov{x}_2}(\ov{x}_0,\ov{y}_3(u))=\pi$. By triangle inequality, $d(x_0,x_3)\leq d(x_0,x_2)+d(x_2,x_3)= d+c$, and it is enough to prove that $$f+2b+c\leq d+2e, \text{ for all }0\leq \gamma \leq u.$$ To do this, note first that at $\gamma=0$ we have $d+2e-(f(0)+2b+c(0))=d+2(b+c(0))-(f(0)+2b+c(0))=d+c(0)-f(0)\geq 0$. In addition, at $\gamma=u$ we have
$d+2e-(f(u)+2b+c(u))=(d+e-f(u)-b)+(d+e-(d+c(u))-b)=\tau(\ov{x}_0,\ov{x}_1,\ov{x}_2,\ov{x}_3(u))+\tau(\ov{x}_0,\ov{x}_1,\ov{x}_2,\ov{y}_3(u))$, which is nonnegative by Corollary \ref{cororef}.

The conclusion follows if we prove that the derivative of the map $\gamma \mapsto f+2b+c-d+2e$ (which is $f_{\gamma}+c_{\gamma}$) does not change sign on the interval $(0,u)$. But, similarly to the computations made in the proof of Corollary \ref{crece}, we obtain 
\begin{equation*}
f_{\gamma}\sinh(f)=-\sinh(a)\sinh(e)\sin(\beta-\gamma), \hspace{3mm} c_{\gamma}\sinh(c)=\sinh(b)\sinh(e)\sin(\gamma),
\end{equation*}
and by the law of sines \eqref{lsin},
\begin{alignat*}{2}
f_{\gamma}+c_{\gamma} & =\sinh(e)\left(\frac{\sinh(b)\sin(\gamma)}{\sinh(c)}-\frac{\sinh(a)\sin(\beta-\gamma)}{\sinh(f)}\right)\\
& = \sinh(e)(\sin(\angle_{\ov{x}_3}(\ov{x}_1,\ov{x}_2)-\sin(\angle_{\ov{x}_3}(\ov{x}_0,\ov{x}_1))))\\
& = \sinh(e)(\sin(\angle_{\ov{y}_3}(\ov{x}_1,\ov{x}_2)-\sin(\angle_{\ov{x}_3}(\ov{x}_0,\ov{x}_1)))). 
\end{alignat*}
By \eqref{lsin} we also have
\begin{equation*}
\frac{\sin(\angle_{\ov{y}_3}(\ov{x}_1,\ov{x}_2))}{\sinh(b)}=\frac{\sin(\angle_{\ov{x}_2}(\ov{x}_1,\ov{y}_3))}{\sinh(e)}\leq \frac{\sin(\angle_{\ov{x}_2}(\ov{x}_1,\ov{y}_3(u)))}{\sinh(e)}=\frac{\sin(\angle_{\ov{y}_3(u)}(\ov{x}_1,\ov{x}_2))}{\sinh(b)},
\end{equation*}
and 
\begin{equation*}
\frac{\sin(\angle_{\ov{y}_3(u)}(\ov{x}_1,\ov{x}_2))}{\sinh(a)}=\frac{\sin(\angle_{\ov{x}_0}(\ov{x}_1,\ov{y}_3(u)))}{\sinh(e)}\leq \frac{\sin(\angle_{\ov{x}_0}(\ov{x}_1,\ov{x}_3))}{\sinh(e)}=\frac{\sin(\angle_{\ov{x}_3}(\ov{x}_0,\ov{x}_1))}{\sinh(a)},
\end{equation*}
where in the first inequality we used $\pi\geq\angle_{\ov{x}_2}(\ov{x}_1,\ov{y}_3)\geq\angle_{\ov{x}_2}(\ov{x}_1,\ov{y}_3(u))=\pi-\angle_{\ov{x}_2}(\ov{x}_0,\ov{x}_1)\geq \pi/2$, and in the second inequality we used $0\leq \angle_{\ov{x}_0}(\ov{x}_1,\ov{y}_3(u))\leq\angle_{\ov{x}_0}(\ov{x}_1,\ov{x}_3) \leq \pi/2$.
We conclude $\sin(\angle_{\ov{x}_3}(\ov{x}_0,\ov{x}_1))\geq \sin(\angle_{\ov{y}_3}(\ov{x}_1,\ov{x}_2)$, implying $f_{\gamma}+c_{\gamma}\leq 0$ for all $0<\gamma<u$ and completing the proof of the proposition. \end{proof}

\begin{figure}[hbt]
\begin{tikzpicture}[scale=0.65]
\draw (2.2408333333,6.6291666666) -- (6,3);
\draw (2.2408333333,6.6291666666) -- (10.985,4.46);
\draw (2.2408333333,6.6291666666) -- (14.760833333,7.129166666);
\draw (2.2408333333,6.6291666666) -- (15.604208,4.249561);
\draw (6,3) -- (10.985,4.46);
\draw (6,3) -- (14.760833333,7.129166666);
\draw (6,3) -- (15.604208,4.249561);
\draw (10.985,4.46) -- (14.760833333,7.129166666);
\draw (10.985,4.46) -- (15.604208,4.249561);
\draw[dashed] (10.985,4.46) -- (14.310313,7.974022);
\draw[dashed] (10.985,4.46) -- (15.680656,3.295147);
\draw[dashed] (10.985,4.46) -- (18.136285,6.554458);
\draw (2.2408333333,6.6291666666) node[left] {\small{$\overline{x}_{0}$}};
\draw (2.2408333333,6.6291666666) node{\small{$\bullet$}};
\draw (6,3) node[below] {\small{$\overline{x}_{1}$}};
\draw (6,3) node{\small{$\bullet$}};
\draw (10.985,4.46) node[below]{\small{$\overline{x}_{2}$}};
\draw (10.985,4.46) node{\small{$\bullet$}};
\draw (14.760833333,7.129166666) node[right] {\small{$\overline{x}_{3}(\gamma)$}};
\draw (14.760833333,7.129166666) node{\small{$\bullet$}};
\draw (15.604208,4.249561) node[right] {\small{$\overline{y}_{3}(\gamma)$}};
\draw (15.604208,4.249561) node{\small{$\bullet$}};
\draw (14.310313,7.974022) node[right] {\small{$\overline{x}_{3}(u)$}};
\draw (15.680656,3.295147) node[right] {\small{$\overline{y}_{3}(u)$}};
\draw[dashed] (15.677194,2.607417) arc (-2.323088:33.237110:9.685154) ;
\draw (9.585353375,4.0500734) arc (16.32418:25.235496:3.735962) ;
\draw (8.8,4.2) node[right] {\small{$\gamma$}};
\end{tikzpicture}
\caption{Proof of $\tau(x_0,x_1,x_2,x_3)\geq -\tau(\overline{x}_0,\overline{x}_1,\overline{x}_2,\overline{x}_3)$: Case 2. }
\end{figure}
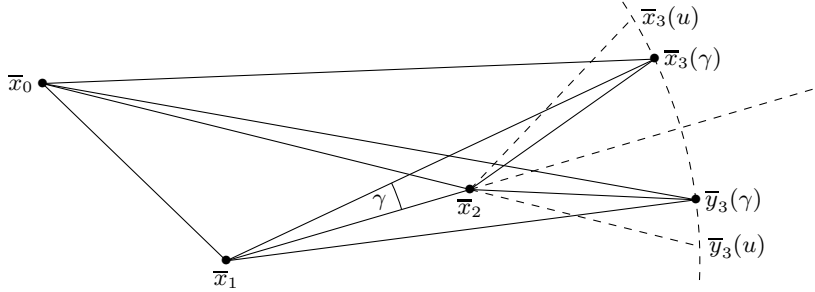

\begin{proof}[Proof of Theorem \ref{catcom}] We will use induction on $n$. The case for $n=3$ is Proposition \ref{cato}, so suppose $n\geq 4$, and assume that the result holds for all good chains of length less than $n$. Let $x_0,\dots,x_n \in X$ be a good chain with comparison chain $\ov{x}_0,\dots,\ov{x}_n \in \Hy^{2}$. 

Since we have the decomposition $$\tau(x_0,\dots,x_n)=\tau(x_0,x_1,x_2,x_3)+\tau(x_0,x_2,x_3,\dots,x_n),$$ it is enough to show that 
$|\tau(x_0,x_2,x_3,\dots,x_n)|\leq \tau(\ov{x}_0,\ov{x}_2,\ov{x}_3,\dots,\ov{x}_n)$. To do this, let $y_0,y_2,y_3,\dots,y_n \in \Hy^{2}$ be the comparison chain for $x_0,x_2,x_3,\dots,x_n$. The chain $x_0,x_2,x_3,\dots,x_n$ is good by Corollary \ref{congood}, so our inductive assumption implies $|\tau(x_0,x_2,x_3,\dots,x_n)|\leq \tau(y_0,y_2,y_3,\dots,y_n)$. In addition we have $d(y_2,y_0)=d(\ov{x}_2,\ov{x}_0)$ and $\groprod{y_{j-1}}{y_{j+1}}{y_j}=\groprod{\ov{x}_j}{\ov{x}_{j-1}}{\ov{x}_{j+1}}$ for $3\leq j\leq n-1$, implying $d(y_{j},y_{j-1})=d(\ov{x}_{j},\ov{x}_{j-1})$ for $3\leq j\leq n$ and $\angle_{y_{j}}(y_{j+1},y_{j-1})=\angle_{\ov{x}_{j}}(\ov{x}_{j+1},\ov{x}_{j-1})$ for $3\leq j\leq n-1$. But we also have $\tau(x_0,x_1,x_2,x_3)\leq \tau(\ov{x}_0,\ov{x}_1,\ov{x}_2,\ov{x}_3 )$, which means $d(\ov{x}_0,\ov{x}_3)\leq d(x_0,x_3)=d(y_0,y_3)$, and hence $\angle_{y_{2}}(y_{3},y_{0})\geq\angle_{\ov{x}_2}(\ov{x}_{3},\ov{x}_{0})$. In this case Lemma \ref{lemachoro} implies $\tau(y_0,y_2,y_3,\dots,y_n)\leq \tau(\ov{x}_0,\ov{x}_2,\ov{x}_3,\dots,\ov{x}_n)$, concluding the proof of the theorem.
\end{proof}

\medskip

\hspace{-4.3mm}\textbf{Acknowledgment}\hspace{2mm} I thank Jairo Bochi for very interesting and valuable discussions and corrections. I was partially supported by CONICYT PIA ACT172001 during the preparation of this article.

\small{Eduardo Oreg\'on-Reyes (\texttt{eoregon@berkeley.edu})}\\
\small{Department of Mathematics}\\
\small{University of California at Berkeley}\\
\small{850 Evans Hall, Berkeley, CA 94720-3860, U.S.A.}

\end{document}